\documentclass[a4paper,12pt]{article}
\usepackage[english]{babel}
\usepackage{amsmath, amsfonts, amssymb, amsthm}
\usepackage{multirow}
\usepackage{graphicx}
\usepackage{fancyhdr}
\usepackage[a-1b]{pdfx}
\usepackage{hyperref}
\usepackage{float}
\usepackage{url}
\usepackage{pdfpages}
\usepackage[utf8]{inputenc}
\usepackage[top=1in, bottom=1.25in, left=1in, right=1 in]{geometry}
\usepackage{amssymb, amsmath}
\usepackage{bbm} % to define \One
\usepackage{xcolor}
\usepackage{graphicx}
\usepackage{adjustbox}

\newcommand{\cA}{\mathcal{A}}
\newcommand{\cB}{\mathcal{B}}

\newcommand{\cF}{\mathcal{F}}
\newcommand{\cG}{\mathcal{G}}

\newcommand{\cN}{\mathcal{N}}

\newcommand{\bF}{\mathbb{F}}
\newcommand{\bG}{\mathbb{G}}

\newcommand{\bH}{\mathbb{H}}

\newcommand{\bOne}{\mathbbm{1}}
\newcommand{\bN}{\mathbb{N}}
\newcommand{\bR}{\mathbb{R}}

\newcommand{\bV}{\mathbb{V}}
\newcommand{\PP}{\boldsymbol{P}}

\newcommand{\EE}{\boldsymbol{E}}
\newcommand{\VV}{\boldsymbol{V}}
\DeclareMathOperator{\sgn}{sgn}

\newcommand{\prT}[1]{(#1_t, \, 0 \leq{t} \leq{T})}
\newcommand{\prTmenos}[1]{(#1_t, \, 0 \leq{t} <{T})}

\newcommand{\overbar}[1]{\mkern 1.5mu\overline{\mkern-1.5mu#1\mkern-1.5mu}\mkern 1.5mu}

\newcommand{\Ito}{It\^{o}\xspace}

\providecommand{\abs}[1]{\lvert#1\rvert}

\providecommand{\keywords}[1]
{
  \small	
  \textbf{\textit{Keywords---}} #1
}
\title{Anticipative information in a Brownian-Poisson market: the binary information}
\author{Bernardo D'Auria
\thanks{UC3M, Department of Statistics, 28911, Legan\'{e}s, Spain \& UC3M-BS, Institute of Financial Big Data, 28903 Getafe, Spain 
  (\email{bernardo.dauria@uc3m.es}).}
\and Jos\'{e} A. Salmer\'{o}n
\thanks{UC3M, Department of Statistics, 28911, Legan\'{e}s, Spain 
  (\email{joseantonio.salmeron@uc3m.es}).}
}

\numberwithin{equation}{section}
\usepackage{color}
\usepackage{amsmath,amsfonts,mathtools}
\usepackage{comment}
\usepackage{multirow} % used in tables
\usepackage{enumitem} % used in enumerate environments
\usepackage{standalone} % used to include independent standalone tex files
\usepackage{etoolbox} % used for recurring commands in tikz pictures
\usepackage{pgfplots} % used for plotting graphs
\usepackage{xspace} % to append a space to customized commands phase
\newtheorem{theorem}{Theorem}[section]
\newtheorem{example}[theorem]{Example}
\newtheorem{lemma}[theorem]{Lemma}
\newtheorem{proposition}[theorem]{Proposition}

\newtheorem{definition}[theorem]{Definition}
\newtheorem{assumption}[theorem]{Assumption}
\newtheorem{remark}[theorem]{Remark}
\providecommand{\citep}[1]{(\cite{#1})}
\newcommand{\email}[1]{\href{mailto:#1}{#1}}

\pgfplotsset{compat=1.17}

\begin{document}
%\includepdf[pages={1,2}]{cover-page.pdf}
\maketitle
\begin{abstract}
The binary information collects all those events that may or may not occur. 
With this kind of variables, a large amount of information can be captured, in particular, about financial assets and their future trends. 
In our paper, we assume the existence of some anticipative information of this type in a market 
whose risky asset dynamics evolve according to a Brownian motion and a Poisson process.
Using Malliavin calculus and filtration enlargement techniques, 
we compute the semimartingale decomposition of the mentioned processes 
and, 
in the pure jump case, we give the exact value of the information. 
Many examples are shown, where the anticipative information is related to some conditions that the constituent processes or their running maximum may or may not verify.
\end{abstract}


\maketitle

\section{Introduction}
In this paper we focus on the study of the presence of some anticipative information in a market composed of a bank bond and a risky asset,
the latter one driven by a Brownian motion 
$W=\prT{W}$ and a compensated Poisson process 
\hbox{$\tilde N = (N_t-\int_0^t\lambda_s ds,\,0\leq t\leq T)$} 
with positive intensity process
\hbox{$\lambda=\prT{\lambda}$}
adapted to the natural filtration generated by
a Poisson process \hbox{$N=\prT{N}$}.
In the optimal portfolio problem, 
a non-informed agent is looking for maximizing her expected logarithmic gains at the end of a trading period \hbox{$T>0$}, 
while playing with the natural information flow $\bF:=\{\cF_t\}_{0\leq t\leq T}$
with $\cF_t:=\sigma(W_s,N_s : 0\leq s \leq t)$.
She will be referred to as the $\bF$-agent.
In addition, we assume that there exists an agent who is informed about a binary random variable~$G$ containing some anticipative information about the path of~$W$ and/or~$N$, i.e., 
she knows if a certain future condition will happen or not. 
The anticipative filtration will be the initial enlargement~\hbox{$\bG := \bF\vee\sigma(G)$} and the agent playing with it will be referred to as the  $\bG$-agent.

Filtration enlargement is a stochastic calculus technique that allows modeling the incorporation of additional non-adapted information. 
It has been initially developed by~\cite{Jacod1985, Jeulin1980} with multiple applications, 
including insider trading or, more general, asymmetric information.
In the seminal paper \cite{AmendingerImkellerSchweizer1998}, 
it is shown that if the dynamics of the risky asset do not include the discontinuous part~$N$, 
then the additional gain is given by the entropy of the random variable~$G$.
After that, much progress has been made in the analysis of the additional information in the Brownian case, 
see~\cite{Amendinger2003, IMKELLER2000, baudoin2002} for the main references.
The research on the Poisson process in the initial enlargement framework started with \cite{ANKIRCHNER2008,Ankirchner11} in which the existence of a compensator is analyzed. 
Although they consider the entropy, %is also taken into account, %he never linked this operator with respect 
the additional gain of an informed~\hbox{$\bG$-agent} in the optimal portfolio problem has not been studied.
In \cite{DiNunnoOksendalProske2006}, a similar framework is studied however they mainly focused on enlargements by the final value of the pair $(W_{T_0},N_{T_0})$ with $T_0\geq T$. 

Our main motivation is to get an expression for the investor's additional gains in a market whose risky asset dynamics depend only on a Poisson process,
in the same spirit of the analysis done in \cite{AmendingerImkellerSchweizer1998}.
This is achieved in Theorem \ref{theo.gain.pois}. %being the entropy relative to the new compensator.
We include various examples about additional information for this case.
Another novelty of our paper is in the kind of examples that we chose to model the additional information.
As it was done with respect to the Brownian motion case, we consider both 
$G = \bOne_{\{N_T\leq b\}}$ and \hbox{$G = \bOne_{\{b_1\leq N_T\leq b_2\}}$}, 
for some constants $b,b_1,b_2 \in~\bN$.
We work also with %the random variable 
\hbox{$G = \bOne_{\{a_1\leq W_T\leq a_2\}\times\{b_1\leq N_T\leq b_2\}}$}, 
where the $\bG$-agent knows if the pair $(W_T,N_T)$ falls within a certain rectangle or not. 
However, as the main example we consider 
\hbox{$G = \bOne_{\{a_1\leq M_T\leq a_2\}\times\{b_1\leq J_T\leq b_2\}}$} being 
\hbox{$M_T:= \sup_{0\leq s \leq T} W_s$} and 
$J_T:= \sup_{0\leq s \leq T}\tilde N_s$ in which the $\bG$-agent knows whether the running maximum processes will be in a certain region or not.
For the majority of our computations, we use Malliavin calculus techniques, we suggest~\cite{DiNunnoOksendalProske2009} for a general overview on this variational calculus.
In~\cite{Privault2003} it was the first time in applying the Malliavin approach to the Poison process in order to prove the existence of conditional densities in some cases. 
In~\cite{Wright2018}, the enlargement of filtration was done without the \emph{Jacod hypothesis} but assuming some Malliavin regularity assumptions on the conditional densities.
Nowadays, the optimal portfolio problem with non-continuous assets is still a topic of research,
for example in~\cite{ChauRunggaldierTankov2018}, 
a weaker arbitrage condition is studied where the authors analyse the additional gains generated by an initial enlargement via super-hedging.

The paper is organised as follows. 
In Section \ref{sec:model-not} we describe the framework %of the problem 
and we introduce the notation.
In Section~\ref{sec:initial-enlarg} we consider the purely jump market and we get the explicit expression of the compensator of the Poisson process for binary random variables. 
The main example of this section is about the terminal value of the Poisson process, that is, $G=\bOne_{\{N_T\in B\}}$ being $B$ an interval.
In Subsection~\ref{subsec:price}, we state Theorem~\ref{theo.gain.pois} in which we get a nice expression for the additional gain of an agent who plays with an initial enlarged filtration.
In Section~\ref{sec:Brow-Pois} we work in a Brownian-Poisson market, deducing similar equations for the compensators. 
The main examples considered in this subsection are~\hbox{$G=\bOne_{\{W_T\in A\}\times \{N_T\in B\}}$} 
and $G=\bOne_{\{M_T\in A\}\times \{J_T\in B\}}$ for $A,B$ some given intervals.
\section{Model and Notation}\label{sec:model-not}
Let $(\Omega, \cF_T,\PP,\bF)$ be a filtered probability space, where the filtration 
\hbox{$\bF = \{\cF_t\}_{0\leq t \leq T}$} is assumed to be complete and right-continuous. 
We assume that the agent is going to invest in a market composed by two assets in a finite horizon time $T>0$. 
The first one is a risk-less bond \hbox{$D=\prT{D}$}
and the second one is a risky stock \hbox{$S=\prT{S}$}. 
The dynamics of both are given by the following SDEs,
\begin{subequations}\label{def.assets}
\begin{align}
    \frac{dD_t}{ D_t} &=\rho_tdt\ ,\quad D_0 = 1\label{def.assets1}\\
    \frac{dS_t}{S_{t-}} &= \mu_tdt + \sigma_tdW_t + \theta_t \left(dN_t-\lambda_t dt\right)\ ,\quad S_0 = s_0 > 0 \ ,\label{def.assets2}
\end{align}
\end{subequations}
where $W=\prT{W}$ is a Brownian motion 
with $\bF^{W}$ its natural filtration and\break
$N = \prT{N}$ is a Poisson process with positive stochastic intensity \hbox{$\lambda = \prT{\lambda}$} adapted to its natural filtration $\bF^{N}$.
Moreover, we assume that 
$$ \int_0^T \lambda_s ds < +\infty\ ,\quad \PP\text{-almost surely}\ . $$
%Let's define the process $\Lambda(s,t) := \int_s^t \lambda_u du$. %, then we compute the probabilities of $N$ as follows
%$$ \PP(N_t-N_s = n \vert \cF_s) = e^{-\Lambda(s,t)}\frac{(\Lambda(s,t))^n}{n!} \ .$$
The filtration $\bF:=\{\cF_t\}_{0\leq t\leq T}$ is generated by the Brownian motion and the Poisson process, and it is augmented by the zero $\PP$-measure sets, $\cN$:
$$\cF_t := \sigma (W_s ,\ N_s:\ 0\leq s \leq t)\vee\mathcal{N}\ .$$
By $\EE$ and $\VV$ we refer the expectation and the variance operators of a given random variable under the measure $\PP$. 
Given a $\sigma$-algebra $\cF$, by~$\EE[\cdot\vert\cF]$ and $\VV[\cdot\vert\cF]$ 
we denote the conditional expectation and the conditional variance.
We define the space $L^2(\Omega,\cF_T,\PP)$,
or simply $L^2(\PP)$ when~\hbox{$(\Omega,\cF_T)$} is clear,
as the set of random variables with finite second moment,
$$ L^2(\PP) = \{F : \EE[F^2] <+ \infty \} \ .$$
We define $L^2(\Omega,\cF_T,dt\times\PP,\bF)$, or simply
$L^2(dt\times\PP)$, as the space of all 
\hbox{$\bF$-adapted} processes,
$$ L^2(dt\times\PP) =  \Big\{X = \prT{X} : \int_0^T \EE\left[X_s^2\right] ds  <+ \infty \Big\}\ . $$
About the market coefficients, in \eqref{def.assets1} and \eqref{def.assets2} we assume that they are 
càglàd~\hbox{$\bF$-adapted} processes that satisfy the following integrability condition
\begin{equation}\label{market.coef.integr}
     \EE\left[\int_0^T\left(\abs{\rho_s}+\abs{\mu_s}+\sigma_s^2+\theta_s^2\right) ds\right] <+\infty\ .
\end{equation}
Using the previous set-up, it is assumed that an agent can control her portfolio by a \textit{self-financing} process $\pi=\prT{\pi}$,
with the aim to maximize her expected logarithmic gains at the finite horizon time.
We denote by~\hbox{$X^\pi=\prT{X^\pi}$} a positive process to model the wealth of the portfolio of the investor under the strategy~$\pi$.
The dynamics of the wealth process are given by the following SDE, for $0 \leq t \leq T$,
\begin{equation}\label{def.X}
    \frac{dX_t^{\pi}}{X_{t-}^{\pi}} = (1-\pi_t) \frac{dD_t}{D_t} + \pi_t \frac{dS_t}{S_{t-}} \ ,\quad X^{\pi}_0=x_0 > 0\ ,
\end{equation}
and by using the evolution of both assets given in \eqref{def.assets} we get
\begin{align*}
    \frac{dX^{\pi}_t}{X^{\pi}_{t-}} =& (1-\pi_t) \rho_tdt + \pi_t \left( \mu_tdt + \sigma_t dW_t+\theta_t\left(dN_t-\lambda_t dt\right) \right)\ ,\quad X^{\pi}_0=x_0\ ,
\end{align*}
where the SDE is well-defined on the probability space $(\Omega, \cF_T,\PP,\bF)$.
Before giving a proper definition of the set of processes $\pi$ that we consider, we look for the natural conditions they should satisfy.
Applying the \Ito formula to the dynamics
of the risky asset given by \eqref{def.assets2}, %is a forward 
we get an explicit solution as follows, 
\begin{align*}
    \ln \frac{S_t}{s_0} =& \int_0^t \left(\mu_s -\frac{1}{2}\sigma^2_s+ \lambda_s(\ln(1+\theta_s) -\theta_s)\right)ds\\
    &+ \int_0^t \sigma_sdW_s + \int_0^t\ln(1+\theta_s)(dN_s-\lambda_s ds) \ ,
\end{align*}
and to ensure that the process $S$ is well-defined, 
we shall assume that,
\begin{equation}\label{hyp1}
    -1<\theta_t\ ,\quad dt\times d\PP-\text{almost surely}.
\end{equation}
If we apply the \Ito formula to the wealth process we get,
\begin{align}\label{solution.X}
\ln \frac{X^{\pi}_t}{x_0} =& \int_0^{t}\left( \rho_s + \pi_s (\mu_s-\rho_s) - \frac{1}{2}\pi^2_s\sigma^2_s +\lambda_s(\ln(1+\pi_s\theta_s)-\pi_s\theta_s)\right)ds\notag \\
&+\int_0^{t}\pi_s\sigma_sdW_s+ \int_0^{t} \ln(1+\pi_s\theta_s) (dN_s-\lambda_s ds)\ ,
\end{align}
provided that these integrals are well-defined. 
To ensure this, we assume the following integrability condition,
\begin{equation}\label{hyp2}
        \EE\left[ \int_0^T\left(\abs{\rho_s}+\abs{\pi_s}\abs{\mu_s-\rho_s}+\pi^2_s\sigma^2_s + \pi^2_s\theta^2_s\right)  ds \right] < +\infty\ .
\end{equation}
Then, in order to guarantee that \hbox{$X^{\pi}$} is well-defined, 
we impose that for all $\pi$ there exists $\epsilon^\pi>0$ such that,
\begin{equation}\label{hyp3}
1+\pi_t\theta_t> \epsilon^{\pi}\ ,\quad 0\leq t\leq T\ .%\\
%    1+\pi_t\theta_t(z) &> \epsilon^{\pi}\\
%    1+\pi_t\kappa_t &> \epsilon^{\pi}\ .
\end{equation}
We assume that the market coefficients satisfy the following relation in order to assure non-infinite expected gains,
\begin{equation}\label{hyp4} 
    \frac{\mu_t-\rho_t}{\theta_t} < 0\ ,\quad dt\times d\PP-\text{almost surely}.
\end{equation}
Now, we can properly define the optimization problem as the supremum of the expected logarithmic gains of the agent's wealth at the finite horizon time $T$.
\begin{equation}\label{optimization.problem}
    J(x_0,\pi):= \EE\left[\ln X^{\pi}_T\vert X^\pi_0 = x_0\right]\ ,\quad \bV_T^{\bH}:= \sup_{\pi\in\cA(\bH)} J(x_0,\pi)\ ,\quad \bF\subseteq\bH\ .
\end{equation}
%where $\bV_T^{\bE}$ is defined as the optimal value of the portfolio at time $T$ given the information flow $\bE$,
%We say that $\pi\in\cA(\bH)$ is a local maximum for the stochastic control problem \eqref{optimization.problem} if $ \EE[\ln  X^{\pi+y\beta}_T]\leq \EE[\ln X^{\pi}_T]\ ,  $ for all bounded $\beta\in\cA(\bH)$ and $\abs{y}<\delta$, for some $\delta > 0$.
Finally we give the definition of the set $\cA(\bH)$ 
made of all admissible strategies for the \hbox{$\bH$-agent}, that is the one playing with information flow $\bH \supseteq \bF$.
\begin{definition}
In the financial market \eqref{def.assets1}-\eqref{def.assets2}, 
we define the set of admissible strategies $\cA(\bH)$ as the set of portfolio processes 
$\pi= \prT{\pi}$ càglàd and adapted with respect to the filtration $\bH$ 
which satisfy the conditions \eqref{hyp1}, \eqref{hyp2}, \eqref{hyp3} and \eqref{hyp4}.
\end{definition}
In the following statement we summarize the results about the optimal portfolios in markets with Brownian noise, Poisson noise or both. 
They can be found in Example~16.22 and Theorem~16.59 of the monography~\cite{DiNunnoOksendalProske2006} 
where it is assumed that the intensity process satisfies $\lambda_t = \lambda > 0,$ $\forall t\in [0,T]$.
\begin{proposition}
If $\sigma_t\neq 0,\,\theta_t = 1$ $dt\times\PP$-almost surely, then the %local maximum 
optimal strategy for the problem \eqref{optimization.problem} with the information flow $\bF$ satisfies
\begin{equation}
    \pi^*_t = \frac{\mu_t-\rho_t-\sigma_t^2-\lambda +\sqrt{(\mu_t-\rho_t-\sigma_t^2-\lambda)^2+4\sigma^2_t(\mu_t-\rho_t)}}{2\sigma^2_t}\ .
\end{equation}
Moreover, if we have that $\sigma_t\neq 0$ and $\theta_t=0$ $dt\times\PP$-almost surely, then we recover the classic Merton problem with the optimal strategy satisfying the relation
\begin{equation}
    \pi_t^* = \frac{\mu_t-\rho_t}{\sigma_t^2}\ .
\end{equation}
Finally, if we have that $\sigma_t = 0$, $\theta_t\neq 0$ and $\lambda\theta_t\neq \mu_t-\rho_t$ $dt\times\PP$-almost surely, then the optimal strategy is given by
\begin{equation}\label{merton.poisson}
    \pi_t^* = \frac{\mu_t-\rho_t}{\lambda\theta_t^2-\theta_t(\mu_t-\rho_t)}\ .
\end{equation}
\end{proposition}
Let $G\in\cF_T$ be a real valued random variable modeling some additional information. 
We introduce the filtration $\bG=\{\cG_t\}_{0\leq t\leq T}$ %that we assume larger than $\bF$, that is  $\bF\subset\bG$. In particular we focus on 
under which the privileged information is accessible since the beginning time $t=0$, that is
\begin{equation}\label{bG.def}
    \cG_t = \bigcap_{s>t}\left( \cF_s \vee \sigma(G) \right)\ ,
\end{equation}
We denote with $\PP^G$ the distribution of $G$, i.e.,
\hbox{$\PP^G(\cdot) = \PP(G\in\cdot)$} 
on $\sigma(G)$ and by $\PP^G(\cdot\vert \cF) = \PP(G\in\cdot\vert \cF)$ 
the corresponding conditional probability with respect to a given~\hbox{$\sigma$-algebra} $\cF$.
In order to assure that any $\bF$-semimartingale is also a~\hbox{$\bG$-semimartingale},
that is known in the literature as the \emph{hypothesis (H')}, see~\cite{JeulinYor1978},
for the rest of this paper we state the following standing assumption, known as  the \emph{Jacod hypothesis}.
\begin{assumption}
%The distribution of the random variable $G$ is positive and $\sigma$-finite while the regular \hbox{$\cF_t$-conditional} distributions
The conditional distributions $\PP^G(\cdot\vert \cF_t)$ for $t\in[0,T)$
almost surely verify the following absolutely continuity condition with respect to $\PP^G$,
$$\PP^G(\cdot\vert \cF_t) \ll \PP^G\ %,\quad 0\leq t < T
.$$
\end{assumption}
This assumption assures the existence of a jointly measurable process \hbox{$p^g = \prTmenos{p^g}$}, 
with $g\in Supp(G)$ such that $\PP(A\vert \cF_t) = \int_A p^g_t \PP^G(dg)$ for any $A\in\sigma(G)$.
In particular, 
\begin{equation*}
    p_t^g =  \frac{\PP^G(dg\vert \cF_t)}{\PP^G(dg)}\ ,\quad 0\leq t<T\ . 
\end{equation*}
In Theorem~2.5 of \cite{Jacod1985} it
is proven that the \emph{Jacod hypothesis} implies the \emph{hypothesis (H')}.
The main results of next section are stated for binary random variables that satisfy the former hypothesis.
As the process~$p^g$ is an~\hbox{$\bF$-martingale} for any $g\in Supp(G)$, see Lemma~2.1 in~\hbox{\cite{AmendingerImkellerSchweizer1998}}, then 
\hbox{$\langle X,p^g\rangle^\bF=\prTmenos{\langle X,p^g\rangle^\bF}$}
is well-defined for any $\bF$-local martingale $X$.
We define,
$$ \langle X,p^G\rangle_s^\bF := \langle X,p^\cdot\rangle_s^\bF\circ G \ .  $$
\begin{proposition}\label{prop.jacod.semimartingale}
Let $X=\prT{X}$ be an $\bF$-local martingale and let~$G$ be an \hbox{$\cF_T$-measurable} random variable satisfying the Jacod hypothesis. 
Then,
\begin{equation}\label{Jacod.semimartingale}
    \widehat{X}_t = X_t - \int_0^t \frac{d\langle X,p^G\rangle_s^\bF}{p^G_{s-}}\ ,\quad 0\leq t \leq T\ ,
\end{equation}
is a $\bG$-local martingale.
\end{proposition}
When the compensator of the process $X$ appearing in \eqref{Jacod.semimartingale} is absolutely continuous with respect to the Lebesgue measure, 
its density is usually called the \emph{information drift} and it plays a crucial role in our computations.
\begin{definition}
The logarithmic price of the information of a filtration $\bH \supset \bF$ is given by
\begin{equation*}
 \Delta \bV_T^\bH = \bV_T^{\bH} - \bV_T^{\bF} \ ,
\end{equation*}
where the quantities on the right-hand side are defined in~\eqref{optimization.problem}.
\end{definition}
To conclude this section, we state the Clark-Ocone formula. 
The operators~$D_t$ and~$D_{t,1}$ refer to the Malliavin derivative in the Brownian and Poisson cases respectively and we consider their generalizations to $L^2(\PP)$.
We refer to \cite{DiNunnoOksendalProske2009} for the details and a general background.
\begin{proposition}
Let $G\in L^2(\PP)$ be an $\cF_T$-measurable random variable, then the following representation holds,
\begin{equation}\label{eq.clark-ocone}
G = \EE[G] + \int_0^T \EE[D_t G\vert \cF_t] dW_t + \int_0^T \EE[D_{t,1} G\vert \cF_t] d\tilde N_t\ .
\end{equation}
\end{proposition}
\section{Initial Enlargements}\label{sec:initial-enlarg}
Using the predictable representation property (PRP) enjoyed by the compensated Poisson process, % in $\cF_T^N$, 
we know that for every $\cF^N_T$-measurable real valued random variable~\hbox{$G\in L^2(\PP)$}, 
there exists an $\bF^N$-adapted process 
$\varphi\in L^2(dt\times\PP)$ such that 
\begin{equation}\label{PRP.G.pois}
    G = \EE[G] + \int_0^T \varphi_s d\tilde N_s\ ,
\end{equation}
usually called the non-anticipative derivative of $G$.
Let $B\in\cB(\bR)$ be a subset and consider the following PRP,
\begin{equation}\label{PRP.B}
    \bOne_{\{G\in B\}} = \PP^G(B) + \int_0^T \varphi_s(B)d\tilde N_s\ ,
\end{equation}
where by the predictable process $\varphi(B)=(\varphi_t(B):0\leq t\leq T)$ we denote the unique one within the Hilbert space $L^2(dt\times\PP)$ that satisfies~\eqref{PRP.B} for a fixed $B$.
We will make the following assumption in order to apply a dominate convergence theorem on $\varphi(\cdot)$.
It can be verify that this assumption holds in all presented examples.
\begin{assumption}
  The process $\varphi:[0,T]\times\cB(\bR) \longrightarrow L^2(dt\times\PP,\bF)$ is bounded \hbox{$\PP$-almost} surely.
\end{assumption}
In the next lemma we prove that $\varphi(\cdot)$ is a vector measure,
we refer to \cite{VectorMeasures77} for the details and a general background on the vector measure theory.
\begin{lemma}\label{lemma.measure.1}
  The set function $B\longrightarrow \varphi(B)$, 
  with $B\in \cB(\bR)$,
  is a countably additive $L^2(dt\times\PP)$-valued vector measure.
\end{lemma}
\begin{proof}
    Let $\{B_i\}_{i=1}^\infty\subset \cB(\bR)$ be a disjoint sequence of subsets satisfying $B = \cup_{i=1}^\infty B_i.$
    Then \hbox{$\bOne_{\{G\in B\}} =  \sum_{i=1}^\infty \bOne_{\{G\in B_i\}}$} $\PP$-almost surely in $L^2(dt\times\PP,\bF)$ (see Example 3 in~\cite{VectorMeasures77}).
    Then, using the PRP we get
    \begin{align*}
        \bOne_{\{G\in B\}}  =  \sum_{i=1}^\infty \bOne_{\{G\in B_i\}}  = \PP(G\in B) + \sum_{i=1}^\infty \int_0^T \varphi_t(B_i)d\tilde N_t =  \int_0^T \sum_{i=1}^\infty\varphi_t(B_i)d\tilde N_t \ ,
    \end{align*}
    where we have used the dominate convergence theorem.
    Finally, by the uniqueness of the PRP we deduce 
    \hbox{$\varphi(B) = \sum_{i=1}^\infty\varphi(B_i)$} 
    and the result follows.
\end{proof}
In the following lemma we state the Radon-Nikodym derivative for the Hilbert valued random measure $\varphi$.
We shall assume that the vector measure $\varphi$ is of bounded variation, i.e., \hbox{$\abs{\varphi(\bR)}<+\infty$} according to Definition~4 of~\cite{VectorMeasures77}.
\begin{lemma}\label{lemma.measure.2}
  %Let $\PP^G$ be the measure induced by the real valued random variable $G$.
  With the previous set-up, there exists a set of processes 
  \hbox{$\psi^g = \prT{\psi^g}$} with \hbox{$g\in Supp(G)$} and within $L^1(dt\times\PP^G)$ such that
  \begin{equation}
      \varphi_t(B) = \int_B \psi_t^g\PP^G(dg)\ ,\quad B\subset \bR\ .
  \end{equation}
\end{lemma}
\begin{proof}
    If $ B\subset \bR$ satisfies $\PP^G(B)= 0$,
    then the random variable $\bOne_{\{G\in B\}}$ is $\PP$-almost surely equal to zero 
    and by the uniqueness of the PRP we know that $\varphi(B) = 0$ and we conclude $\varphi \ll \PP^G$ on $\sigma(G)$.
    We get the result by applying Proposition 2.1 of~\cite{Kakihara11}.
\end{proof}
We fix $g$ in the support of $G$, not yet necessary binary. By Lemmas~\ref{lemma.measure.1} and~\ref{lemma.measure.2} we know that 
there exists a set of processes 
\hbox{$\psi^g = \prT{\psi^g}$} with \hbox{$g\in Supp(G)$} such that,
\begin{equation}\label{repr.binary}
    \bOne_{\{G\in dg\}} =%&= \PP(G\in dg) + \int_0^T \varphi^{g}_t d\tilde N_t = 
    \PP^G(dg) + \int_0^T \psi_s^g\PP^G(dg) d\tilde N_s\ . 
\end{equation}
When $G$ is purely atomic, the PRP \eqref{repr.binary} reduces to
\begin{equation}\label{repr.binary.atomic}
    \bOne_{\{G = g\}} =%&= \PP(G\in dg) + \int_0^T \varphi^{g}_t d\tilde N_t = 
    \PP(G = g) + \int_0^T \psi_s^g\PP(G = g) d\tilde N_s\ .
\end{equation}
\begin{lemma}\label{pois.drift}
  Let $G\in L^2(\PP)$ be an $\cF_T^N$-measurable random variable satisfying the {Jacod hypothesis}, 
  then the process $\gamma^G=\prTmenos{\gamma^G}$ defined as
    \begin{equation}
        \gamma_t^G :=  \frac{\psi^G_t}{p_t^G}\ ,%\quad g\in \text{Supp}(G)
        %\gamma_t^g :=  \frac{\color{red}\psi^g_t\PP(G\in dg)}{\PP(G\in dg\vert \cF^N_t)}\ ,\quad g\in \text{Supp}(G)
        \end{equation}
        satisfies that $ N_\cdot-\int_0^\cdot\lambda_s\left(1+\gamma_s^G\right) ds$
        is a $\bG$-local martingale.
    \end{lemma}
\begin{proof}
    As before, we compute the process $p_t^g$,
\begin{align}
        p_t^g &= \frac{\PP^G(dg\vert \cF^N_t)}{\PP^G(dg)} = \frac{\EE[\bOne_{\{G\in dg\}}\vert \cF^N_t]}{\PP^G(dg)} \notag\\ 
        &= \frac{\PP^G(dg) +  \int_0^t \psi_s^g\PP^G(dg)d\tilde N_s}{\PP^G(dg)}
        = 1 + \frac{\int_0^t \psi_s^g\PP^G(dg) d\tilde N_s}{\PP^G(dg)}\ .
\end{align}
    Differentiating in a \Ito sense, we get 
    $dp_t^g =\psi^{g}_t d\tilde N_t$.
    According to Theorem 9.2.2.1 of~\cite{Jeanblanc2009mathematical}, we have
    $$ \langle p^g, \tilde N\rangle_t^\bF = \int_0^t \psi^{g}_s  d\langle \tilde N, \tilde N\rangle_s^\bF =  \int_0^t \psi^{g}_s  \lambda_s ds\ . $$
    Then, following the lines of Proposition \ref{prop.jacod.semimartingale}, we get
    $$ \int_0^t \frac{d\langle p^G, \tilde N\rangle_s^\bF}{p_{s-}^G} =   \int_0^t \frac{\psi^G_s}{p_{s-}^G} \lambda_s ds\ ,\quad 0\leq t < T\ ,$$
    and the result holds true.
\end{proof}
Note that the Lemma \ref{pois.drift} simplifies the computations leading to Theorem~2.4 of~\cite{ANKIRCHNER2008}.
From here until the end of the section, we will assume that the random variable $G$ is binary, i.e., $Supp(G)=\{0,1\}$.
\begin{theorem}\label{alpha.binary}
  If $G$ is a binary random variable, then
  \begin{equation}
      \gamma_t^G = \varphi_t\frac{G-\EE[G\vert \cF^N_t]}{\VV[G\vert \cF^N_t]}\ .
  \end{equation}
\end{theorem}
\begin{proof}
As $G = \bOne_{\{G=1\}}$, by the uniqueness of the representation we conclude 
\hbox{$\varphi = \varphi^1$} and \hbox{$\varphi^1 = -\varphi^0$}.
Using that 
$$\EE[G\vert \cF^N_t] = \PP(G=1\vert \cF^N_t)\ ,\quad \VV[G\vert \cF^N_t] = \PP(G=1\vert \cF^N_t)(1-\PP(G=1\vert \cF^N_t))\ ,$$ 
the result follows by applying \eqref{PRP.G.pois} and \eqref{repr.binary.atomic}.
\end{proof}
\begin{remark}
Note that we can express
$$   \gamma_t^g= \frac{\varphi_t}{\PP(G =  0\vert \cF^N_t)-g} \ ,\quad g\in\{0,1\}\ ,$$
where $\varphi$ is the non-anticipative derivative of the binary random variable $G$.
\end{remark}
By using the Clark-Ocone formula we can deduce that 
$\varphi_t = \EE[D_{t,1} G\vert \cF^N_t],$ \hbox{$\forall t \in [0,T]$} and $\PP$-almost surely,
which allows to compute some interesting examples.
Following~\cite{SoleUtzetVives2007}, we introduce the following operator 
\begin{equation}
    \Psi_{t,1} G := {G(\omega_{(t,1)})-G(\omega)}\ ,
\end{equation}
being $\omega_{(t,1)}$ the modification of the trajectory $\omega$ by adding a new jump of size~$1$ at time $t$. 
In \cite{AlosLeonPontierVives2008} it is proved that if $\abs{\Psi_{t,1} G}^2\in L^1(dt \times \PP)$, then this operator coincides with the usual Malliavin derivative, in the case of the Poisson process we have that \hbox{$\Psi_{t,1} G = D_{t,1} G$}.
%%\subsubsection{Half-bounded interval}
\begin{example}\label{Example.1}
Let $G = \bOne_{\{N_T \leq b \}} $ with $b\in\bN$ and consider the initial enlargement \hbox{$\bG\supset\bF$}.
We compute the process $\varphi$ as follows,
\begin{equation*}
    \Psi_{t,1} \bOne_{\{N_T \leq b \}} = \bOne_{\{N_T + 1 \leq b \}} - \bOne_{\{N_T \leq b \}} = -\bOne_{\{N_T = b \}}
\end{equation*}
which obviously satisfies the integrability condition and therefore 
$$D_{t,1} \bOne_{\{N_T \leq b \}} = %\bOne_{\{N_T + 1 \leq b \}} - \bOne_{\{N_T \leq b \}} =
-\bOne_{\{N_T = b \}}\ . $$
In order to compute the Clark-Ocone formula we compute its conditional expectation as follows,
\begin{equation*}
    \EE\left[D_{t,1} \bOne_{\{N_T \leq b \}} \vert \cF_t^N\right] = -\PP\left(N_T = b \vert \cF_t^N\right)\ .%  = -e^{-\Lambda(t,T)}\frac{(\Lambda(t,T))^{b-N_t}}{(b-N_t)!} \bOne_{\{N_t\leq b\}}
\end{equation*}
The PRP holds true,
\begin{equation}
    \bOne_{\{N_T \leq b \}} = \PP(N_T \leq b) - \int_0^T \PP\left(N_T = b \vert \cF_t^N\right) %\bOne_{\{N_t\leq b\}}
    d\tilde N_t\ . 
    %\bOne_{\{N_T \leq b \}} = \PP(N_T \leq b) - \int_0^T e^{-\lambda_t}\frac{\lambda_t^{b-N_t}}{(b-N_t)!}  \bOne_{\{N_t\leq b\}} d\tilde N_t\ . 
\end{equation}
%Now, we analyze the term $(G-\EE[G\vert \cF_t])/{\VV[G\vert \cF_t]}$. Conditioned to $\{G=1\}$, it is reduced to $1/\PP(N_T\leq b \vert \cF_t)$ and conditioned to $\{G=0\}$ is $-1/\PP(N_T > b \vert \cF_t)$.
Then, we compute the compensator as
\begin{equation}
     %\gamma^g_t = (-1)^g\frac{\PP\left(N_T = b \vert \cF_t^N\right)}{\PP(G = g\vert \cF_t^N)}\bOne_{\{N_t\leq b\}}
    \gamma^g_t = \frac{\PP\left(N_T = b \vert \cF_t^N\right)}{\PP(N_T > b\vert \cF_t^N)-g} \ ,\quad g\in\{0,1\}\ .
\end{equation}
Note that $\gamma^G_t \geq -1$ because $\gamma^0_t \geq 0$ and
$$ \gamma^1_t = - \frac{\PP(N_T = b\vert \cF_t^N)}{\PP(N_T \leq b \vert \cF_t^N)}% \bOne_{\{N_t\leq b\}}
\geq -1\ .$$ 
We will need this in order to compute $\ln(1 + \gamma_t^G)$.
In addition, we can achieve more explicit results if we assume that $\lambda_t$ is $\cF_0^N$-measurable $\forall t\in[0,T]$, 
where the \hbox{$\sigma$-algebra} $\cF_0^N$ can be non-constant but we need to assure that $N$ is still an~\hbox{$\bF^N$-adapted} counting process with compensator $\lambda$.
Then we can compute the probabilities as follows
$$ \PP(N_t-N_s = n \vert \cF^N_s) = e^{-\Lambda(s,t)}\frac{(\Lambda(s,t))^n}{n!} \ ,\quad \Lambda(s,t):= \int_s^t \lambda_u du\ .$$
In particular, we are interested in 
$$ \PP\left(N_T = b \vert \cF_t^N\right) = e^{-\Lambda(t,T)}\frac{(\Lambda(t,T))^{b-N_t}}{(b-N_t)!} \bOne_{\{N_t\leq b\}} \ .$$
The PRP is simplified as follows
\begin{equation}
    \bOne_{\{N_T \leq b \}} = \PP(N_T \leq b) - \int_0^T e^{-\Lambda(t,T)}\frac{(\Lambda(t,T))^{b-N_t}}{(b-N_t)!}  \bOne_{\{N_t\leq b\}} d\tilde N_t 
\end{equation}
and the compensator is
\begin{equation}
    \gamma^G_t = \frac{e^{-\Lambda(t,T)}}{(b-N_t)!}\frac{(\Lambda(t,T))^{b-N_t}\bOne_{\{N_t\leq b\}}}{\PP(N_T > b\vert \cF_t^N)-G} \ .%\ ,\quad g\in\{0,1\}\ .
\end{equation}
Note that, in the simplest case of time-homogeneous Poisson process with constant intensity $\lambda > 0$, we have $\Lambda(t,T) = \lambda(T-t)$.
\end{example}
\begin{example}
Let $G = \bOne_{\{ N_T \in B \}} $ with $B=[b_1,b_2]$, and $b_1,b_2\in\bN$. 
We consider the initial enlargement $\bG\supset\bF$.
We compute the process $\varphi$ as before,
$$D_{t,1} \bOne_{\{ N_T \in B\}} = \bOne_{\{ N_T+1 \in B\}} - \bOne_{\{ N_T \in B\}}\ .$$
In order to compute the Clark-Ocone formula we compute its conditional expectation as follows,
\begin{align*}
    \EE\left[D_{t,1} \bOne_{\{ N_T \in B\}} \vert \cF_t^N\right] &= \PP\left(N_T = b_1-1\vert \cF^N_t \right) - \PP\left(N_T = b_2\vert \cF_t^N \right)
%    \EE\left[D_{t,1} \bOne_{\{ N_T \in B\}} \vert \cF_t^N\right] &= \PP\left(N_{T-t} = b_1-N_t-1 \right) - \PP\left(N_{T-t} = b_2-N_t \right)\\
%     &=  -e^{-\lambda_t}\left(  \frac{\lambda_t^{b_2-N_t}}{(b_2-N_t)!} - \frac{\lambda_t^{b_1-N_t-1}}{(b_1-N_t-1)!} \right)\bOne_{\{N_t\leq b_2\}}
\end{align*}
and the PRP holds,
\begin{equation}
    \bOne_{\{N_T\in B \}} = \PP(N_T\in B) - \int_0^T  \left(\PP\left(N_T = b_2\vert \cF_t^N \right) - \PP\left(N_T = b_1-1 \vert  \cF^N_t \right) \right)  d\tilde N_t\ ,
    %\bOne_{\{N_T\in B \}} = \PP(N_T\in B) - \int_0^T  e^{-\lambda_t}\left(\frac{\lambda_t^{b_2-N_t}}{(b_2-N_t)!} - \frac{\lambda_t^{b_1-N_t-1}}{(b_1-N_t-1)!} \right) \bOne_{\{N_t\leq b_2\}} d\tilde N_t\ ,
\end{equation}
giving the following formula for compensator
\begin{equation}
    \gamma^G_t = \frac{ \PP\left(N_T = b_2\vert \cF_t^N \right) - \PP\left(N_T = b_1-1 \vert  \cF^N_t \right)}{\PP(N_T\in B^{c}\vert \cF_t^N)-G} \ .
\end{equation}
A direct computation shows that $\gamma^G_t\geq -1.$
If we assume that the process $\lambda$ is \hbox{$\cF_0^N$-measurable}, 
then the PRP simplifies as follows
\begin{align*}
    \bOne_{\{N_T\in B \}} =& \,\PP(N_T\in B)\\ 
    &- \int_0^T  e^{-\Lambda(t,T)}\left(\frac{(\Lambda(t,T))^{b_2-N_t}}{(b_2-N_t)!}\bOne_{\{N_t\leq b_2\}} - \frac{(\Lambda(t,T))^{b_1-N_t-1}}{(b_1-N_t-1)!}\bOne_{\{N_t< b_1\}} \right) d\tilde N_t
\end{align*}
and the compensator is
$$ 
 \gamma^G_t = \left(\frac{(\Lambda(t,T))^{b_2-N_t}}{(b_2-N_t)!}\bOne_{\{N_t\leq b_2\}} - \frac{(\Lambda(t,T))^{b_1-N_t-1}}{(b_1-N_t-1)!}\bOne_{\{N_t < b_1\}} \right)  \frac{ e^{-\Lambda(t,T)}}{\PP(N_T\in B^{c}\vert \cF_t^N)-G} \ .
$$
\end{example}
\subsection{Price of the information}\label{subsec:price}
Working in the filtration $\bF$, 
if we take expectation in \eqref{solution.X}, then
$$ \EE\left[\ln\frac{X^\pi_T}{X_0}\right] = \EE\left[\int_0^T \rho_s + \pi_s(\mu_s-\rho_s) + \lambda_s(\ln(1+\pi_s\theta_s)-\pi_s\theta_s)\, ds\right]\ , $$
with $\pi\in\cA(\bF)$. 
Using that the maximum is attained in the strategy given by~\eqref{merton.poisson}, the solution of the optimal control problem is
%$$ \bV^{\bF}_T = \int_0^T \EE\left[\rho_s + \frac{(\mu_s-\rho_s)^2}{\lambda_s\theta_s^2 - \theta_s(\mu_s-\rho_s)} - \lambda_s \frac{\mu_s-\rho_s}{\lambda_s\theta_s-(\mu_s-\rho_s)} + \lambda_s\ln\left(\frac{\lambda_s\theta_s}{\lambda_s\theta_s - (\mu_s-\rho_s)}\right) \right]ds\ .$$
\begin{equation}
    \bV^{\bF}_T = \int_0^T \EE\left[\rho_s - \frac{\mu_s-\rho_s}{\theta_s} + \lambda_s\ln\left(\frac{\lambda_s}{\lambda_s - (\mu_s-\rho_s)/{\theta_s}}\right) \right]ds\ ,
\end{equation}
which is trivially positive as for the non-arbitrage condition \eqref{hyp4} all the terms are well-defined and positive.
If $\pi\in\cA(\bG)$, the \Ito integral with respect to $\tilde N$ is not necessary well defined, 
but using the \emph{Jacod hypothesis} and, therefore, 
the $\bG$-semimartingale decomposition we can still use it.
We define the process
$$ \widehat{N}_t := N_t - \int_0^t \lambda_s\left(1+\gamma_s^G\right) ds\ ,\quad 0\leq t\leq T\ , $$
which is a $\bG$-local martingale by Theorem \ref{alpha.binary}.
Then the dynamics of the wealth process satisfy the following SDE,
\begin{equation}\label{new.dynamics}
        \frac{dX_t^\pi}{X_{t-}^{\pi}} = \left((1-\pi_t)\rho_t+ \pi_t\mu_t+ \pi_t\theta_t\lambda_t\gamma_t^G \right)dt + \pi_t\theta_t d\widehat{N}_t \ ,\quad X_0 = x_0 \ ,
\end{equation}
%    where $d^-$ denotes the differential with respect to the forward interpretation.     For a background and more details about the forward integral we refer to the Chapter 15 of the monography \cite{DiNunnoOksendalProske2009}.    Using Theorem 15.8 of \cite{DiNunnoOksendalProske2009}, we can apply a version of the \Ito formula for forward integrals, by getting the following explicit solution 
and we have the following explicit solution
\begin{align*}
    \ln \frac{X_t^\pi}{x_0} =& \int_0^t\left( \rho_s + \pi_s(\mu_s-\rho_s) + \lambda_s(1+\gamma_s^G)\ln(1+\pi_s\theta_s)-\lambda_s\pi_s\theta_s\right)ds \\
    &+ \int_0^t \ln(1+\pi_s\theta_s) d\widehat{N}_s\ .
\end{align*}
%    According to Lemma \ref{pois.drift}, the process $\gamma^G = \prT{\gamma^G}$ verifies that \hbox{$\widehat{N}_t= \tilde N_t - \lambda\int_0^t\gamma_s^G ds$} is a $\bG$-local martingale. 
As it is argued in \cite{AmendingerImkellerSchweizer1998}, by using the integrability condition of $\ln(1+\pi_t\theta_t)$, the stochastic integral satisfies
    $$ \EE\left[\int_0^T \ln(1+\pi_s\theta_s)d \widehat{N}_s\right] = 0 \ .$$
    Then,
    \begin{equation}\label{wealth.G}
        \EE\left[\ln \frac{X_T^\pi}{x_0}\right] = \int_0^T \EE\left[\rho_s + \pi_s(\mu_s-\rho_s) + \lambda_s(1+\gamma_s^G)\ln(1+\pi_s\theta_s)-\lambda_s\pi_s\theta_s\right]ds \ .
    \end{equation}
In the next proposition we compute the optimal strategy for a $\bG$-agent.
\begin{proposition}\label{pi.pois}
If G is binary and $\theta_t\neq 0$ $dt\times \PP$-almost surely, 
then the strategy solving the optimization problem \eqref{optimization.problem} with information flow $\bG$ is given by
\begin{equation}\label{pi.pois.eq}
     \pi_t^G = \frac{\mu_t-\rho_t}{\lambda_t\theta_t^2-\theta_t(\mu_t-\rho_t)} + \frac{\lambda_t\gamma^G_t}{\lambda_t\theta_t-(\mu_t-\rho_t)}\ .
\end{equation}
%Moreover, on the set that $\lambda_t\theta_t = \mu_t-\rho_t$ holds true, there exists no optimal strategy.
\end{proposition}
\begin{proof}
    We apply a standard perturbation argument. 
    Let $\beta\in\cA(\bG)$ be a bounded strategy and let $\epsilon >0$. 
    Then we define
    $$ I(\epsilon) =\int_0^T\EE\left[ %\rho_s +
    (\pi_s+\epsilon\beta_s)(\mu_s-\rho_s) + \lambda_s(1+\gamma_s^G)\ln(1+(\pi_s+\epsilon\beta_s)\theta_s)-\lambda_s(\pi_s+\epsilon\beta_s)\theta_s\right]ds. $$
    We impose $I'(0) = 0$ in order to get the optimality condition. 
    \begin{align*}
        0 &= I'(0) = \lim_{\epsilon\to 0}\frac{I(\epsilon)-I(0)}{\epsilon} \\
        &=  \lim_{\epsilon\to 0}\frac{\int_0^T \EE\left[\epsilon\beta_s(\mu_s-\rho_s)  +\lambda_s(1+\gamma_s^G)\ln\left(1+\epsilon\frac{\beta_s\theta_s}{1+\pi_s\theta_s}\right) -\epsilon\lambda_s\beta_s\theta_s\right] ds}{\epsilon}\\
        &= \int_0^T\EE\left[\beta_s(\mu_s-\rho_s)  +\lambda_s(1+\gamma_s^G)\frac{\beta_s\theta_s}{1+\pi_s\theta_s} -\lambda_s\beta_s\theta_s\right] ds\ .
    \end{align*}
    We take $\beta_s = \xi \bOne_{\{t\leq s < t+h\}}$ being $\xi$ a bounded and $\cG_t$-measurable random variable and $I'(0)$ can be rewritten as follows,
    \begin{equation*}
       0 = I'(0) = \int_t^{t+h}\EE\left[\xi\left(\mu_s-\rho_s  +\frac{\lambda_s(1+\gamma_s^G)\theta_s}{1+\pi_s\theta_s} -\lambda_s\theta_s\right)\right] ds\ .
    \end{equation*}
    Then, we conclude that
    \begin{equation*}
       0 = \int_t^{t+h}\EE\left[\mu_s-\rho_s  +\frac{\lambda_s(1+\gamma_s^G)\theta_s}{1+\pi_s\theta_s} -\lambda_s\theta_s\vert \cG_t\right]  ds 
    \end{equation*}
    where the last step is due that $\xi$ is any $\cG_t$-measurable random variable.
    Taking $h\to 0$ we get the condition
    $$0 = \mu_s-\rho_s  +\frac{\lambda_s(1+\gamma_s^G)\theta_s}{1+\pi_s\theta_s} -\lambda_s\theta_s $$
    and the result holds true.
\end{proof}
Note that the strategy \eqref{pi.pois.eq} is well-defined thanks to the no-arbitrage assumption given in~\eqref{hyp4}.
\begin{lemma}
The information drift is orthogonal to the information flow $\bF$, i.e.,
$$ \EE\left[\gamma^G_t\vert \cF_t\right] = 0 $$
\end{lemma}
\begin{proof}
    It comes directly from Theorem \ref{alpha.binary}.
\end{proof}
\begin{theorem}\label{theo.gain.pois}
Let $\bG\supset\bF$ be the initial enlargement with $G$ a binary random variable, then,
\begin{equation}\label{gain.entropy}
\bV_T^{\bG} - \bV_T^{\bF} = \int_0^T \EE\left[   \lambda_s(1+\gamma_s^G) \ln\left(\lambda_s(1+\gamma_s^G) \right) -\lambda_s\ln\lambda_s\right] ds \geq 0 \ .
\end{equation}
\end{theorem}
\begin{proof}
    By \eqref{wealth.G} we have
    \begin{equation}\label{gain.G}
        \bV_T^\bG = \int_0^T \EE\left[ \rho_s + \pi_s^G(\mu_s-\rho_s)  + \lambda_s(1+\gamma_s^G)\ln(1+\pi_s^G\theta_s) -\lambda_s\pi_s^G\theta_s\right] ds\ ,
    \end{equation}
    being $\pi^G$ the process defined in the Proposition \ref{pi.pois}.
    We compute the following terms
    \begin{align*}
    \pi_s^G\left(\mu_s-\rho_s-\lambda_s\theta_s\right) &= -\lambda_s\gamma_s^G - \frac{\mu_s-\rho_s}{\theta_s}\\
%        \pi_s^G(\mu_s-\rho_s) &= \frac{(\mu_s-\rho_s)^2}{\lambda_s\theta_s^2-\theta_s(\mu_s-\rho_s)} + \lambda_s\gamma_s^G\frac{\mu_s-\rho_s}{\lambda_s\theta_s-(\mu_s-\rho_s)}\ ,\\
    \ln\left(1+\pi_s^G\theta_s\right)&= \ln\left(  \frac{\lambda_s(1+\gamma_s^G)}{\lambda_s-(\mu_s-\rho_s)/\theta_s} \right)\ .
%    \lambda_s\pi_s^G\theta_s &=  \lambda_s\frac{\mu_s-\rho_s}{\lambda_s\theta_s-(\mu_s-\rho_s)}+ \lambda_s^2\gamma_s^G\frac{\theta_s}{\lambda_s\theta_s-(\mu_s-\rho_s)}\ .
    \end{align*}
    By substituting them in \eqref{gain.G} we have the following expression, 
    \begin{align*}
    \bV_T^\bG =& \int_0^T \EE\left[\rho_s - \lambda_s\gamma_s^G - \frac{\mu_s-\rho_s}{\theta_s} +\lambda_s(1+\gamma_s^G)  \ln\left(  \frac{\lambda_s(1+\gamma_s^G)}{\lambda_s-(\mu_s-\rho_s)/\theta_s} \right)  \right] ds
%         \bV_T^\bG =& \int_0^T \EE\left[\rho_s + \frac{(\mu_s-\rho_s)^2}{\lambda_s\theta_s^2-\theta_s(\mu_s-\rho_s)} +\lambda_s \gamma_s^G\frac{\mu_s-\rho_s}{\lambda_s\theta_s-(\mu_s-\rho_s)}  \right.\\
%         &+ \left.   \lambda_s(1+\gamma_s^G) \ln\left( \frac{\lambda_s (1+\gamma_s^G )\theta_s}{\lambda_s\theta_s-(\mu_s-\rho_s)} \right) - \lambda_s\frac{\mu_s-\rho_s}{\lambda_s\theta_s-(\mu_s-\rho_s)} \right.\\
%         &- \left. \frac{\lambda_s^2\gamma_s^G\theta_s}{\lambda_s\theta_s-(\mu_s-\rho_s)}\right] ds\ .% \textcolor{red}{- \lambda_s\ln\left(\frac{\lambda_s\theta_t}{\lambda_s\theta_t - (\mu_t-\rho_t)}\right) } \right] dt\ .
    \end{align*}
    Finally, we compute the difference as follows
    \begin{align*}
    \bV_T^\bG-\bV_T^\bF=&\int_0^T \EE\left[-  \lambda_s\gamma_s^G +\lambda_s(1+\gamma_s^G)  \ln\left(  \frac{\lambda_s(1+\gamma_s^G)}{\lambda_s-(\mu_s-\rho_s)/\theta_s} \right)\right.\\
    &\left.-\lambda_s \ln\left(  \frac{\lambda_s}{\lambda_s-(\mu_s-\rho_s)/\theta_s} \right) \right] ds\\
    =& \int_0^T \EE\left[  \lambda_s\gamma_s^G +\lambda_s(1+\gamma_s^G)  \ln\left(  1+\gamma_s^G \right) \right] ds \geq 0
%        \bV_T^\bG-\bV_T^\bF=&\int_0^T \EE\left[  \frac{\lambda_s\gamma_s^G(\mu_s-\rho_s)}{\lambda_s\theta_s-(\mu_s-\rho_s)} + \lambda_s(1+\gamma_s^G) \ln\left( \frac{\lambda_s (1+\gamma_s^G )\theta_s}{\lambda_s\theta_s-(\mu_s-\rho_s)} \right) \right.\\
%        &  \left.-\frac{\lambda_s^2\gamma_s^G\theta_s}{\lambda_s\theta_s-(\mu_s-\rho_s)} - \lambda_s\ln\left(\frac{\lambda_s\theta_s}{\lambda_s\theta_s - (\mu_s-\rho_s)}\right)  \right] ds\ .
    \end{align*}
    where we have applied the tower property of the conditional expectation in order to simplify the expression as follows
    \begin{align*}
        \EE\left[\gamma_s^G\ln\left(  \frac{\lambda_s}{\lambda_s-(\mu_s-\rho_s)/\theta_s} \right) \right] &= \EE\left[\EE\left[\gamma_s^G\ln\left(  \frac{\lambda_s}{\lambda_s-(\mu_s-\rho_s)/\theta_s} \right) \vert \cF_s\right]\right] \\
        &=  \EE\left[\ln\left(  \frac{\lambda_s}{\lambda_s-(\mu_s-\rho_s)/\theta_s} \right) \EE\left[\gamma_s^G\vert \cF_s\right]\right] = 0\ .
    \end{align*}
    It can be checked that the function $h(x,y)=-xy+x(1+y)\ln(1+y)$ is positive when $x>0$ and $y>-1$ and we get the non-negativity of the additional gains.
    Finally by applying again the tower property to the term $-\lambda_s\gamma_s^G$ and by adding and subtracting the term $\lambda_s(1+\gamma_s^G)\ln \lambda_s$ we get the result.
%    \begin{align*}
%        \bV_T^\bG-\bV_T^\bF &= \int_0^T \EE\left[   \lambda_s(1+\gamma_s^G) \ln\left( \frac{\lambda_s(1+\gamma_s^G)\theta_s}{\lambda_s\theta_s-(\mu_s-\rho_s)} \right) -\lambda_s\ln\left(\frac{\lambda_s\theta_s}{\lambda_s\theta_s - (\mu_s-\rho_s)}\right) \right] ds\\
%        &= \int_0^T \EE\left[   \lambda_s(1+\gamma_s^G) \ln\left(1+\gamma_s^G \right)+\lambda_s\gamma_s^G \ln\left( \frac{\theta_s\lambda_s}{\lambda_s\theta_s-(\mu_s-\rho_s)} \right) \right] ds\\ %
%        &= \int_0^T \EE\left[   \lambda_s(1+\gamma_s^G) \ln\left(1+\gamma_s^G \right) \right] ds \ ,
%    \end{align*}
%    where in the last step we have used again the tower propery and the zero conditional expectation of $\gamma_s^G$.
\end{proof}
\begin{remark}
Note that Theorem \ref{theo.gain.pois} holds true not only for binary random variables but for any random variable $G$ that satisfies $\EE[\gamma_t^G\vert \cF_t] = 0 $.
\end{remark}
\section{Mixed Brownian-Poisson market}\label{sec:Brow-Pois}
%Using the PRP that enjoys the compensated Poisson process in $\cF_T^N$, we know that for $G\in\cF^N_T\cap L^2(\PP)$, there exists an $L^2(\PP\times dt)$-integrable process $\varphi=\prT{\varphi}$ such that
Using the Clark-Ocone formula stated in Equation \eqref{eq.clark-ocone}, 
we can apply the same representation formula for every $\cF_T$-measurable random variable $G\in L^2(\PP)$,
$$ G = \EE[G] + \int_0^T \phi_s d W_s + \int_0^T \varphi_s d\tilde N_s\ , $$
where the processes $\phi$ and $\varphi$ are related to the Malliavin derivative.
Let \hbox{$B\in\cB(\bR)$} be a subset and we consider the following PRP
\begin{equation}\label{PRP.B.mixed}
    \bOne_{\{G\in B\}} = \PP^G(B) + \int_0^T \phi_s(B)dW_s + \int_0^T \varphi_s(B)d\tilde N_s\ .
\end{equation}
\begin{assumption}
  The processes $\phi,\varphi:[0,T]\times\cB(\bR) \longrightarrow L^2(dt\times\PP,\bF)$ are bounded \hbox{$\PP$-almost} surely.
\end{assumption}
We fix $g$ in the support of $G$, not necessary binary yet,
by reasoning as in the beginning of Section \ref{sec:initial-enlarg}, 
and we know that there exist two $\bF$-adapted processes $\zeta^g,\,\psi^g$ with $g\in Supp(G)$ such that,
\begin{equation}\label{PRP.binary.mixed}
    \bOne_{\{G\in dg\}} 
    = \PP^G(dg)+ \int_0^T \zeta^g_s \PP^G(dg) dW_s + \int_0^T \psi^{g}_s \PP^G(dg) d\tilde N_s\ . 
\end{equation}
where $\phi^{g}_t=\zeta_t^g\PP(G=g)$ and
$\varphi^{g}_t=\psi_t^g\PP(G=g)$ in the PRP \eqref{PRP.binary.mixed} when $G$ is purely atomic.
\begin{lemma}\label{mixed.compens}
  Let $G\in L^2(\PP)$ be an $\cF_T$-measurable random variable satisfying the \emph{Jacod hypothesis},
  then the processes $\alpha^G$ and $\gamma^G$ defined as
    \begin{equation}
        \alpha_t^G := \frac{\zeta^{G}_t}{p_t^G}\ ,\quad
        \gamma_t^G := \frac{\psi^{G}_t}{p_t^G} \ ,\quad 0\leq t < T\ ,
        \end{equation}
        satisfy that $W_{\cdot}-\int_0^{\cdot} \alpha^G_s ds$ and $ N_{\cdot} - \int_0^{\cdot}\lambda_s(1+\gamma_s^G) ds$ are $\bG$-local martingales.
    \end{lemma}
\begin{proof}
    As before, we compute the process $p_t^g$,
\begin{align*}
        p_t^g &= \frac{\PP^G(dg\vert \cF_t)}{\PP^G(dg)} = \frac{\EE[\bOne_{\{G\in dg\}}\vert \cF_t]}{\PP^G(dg)} = 1 + \frac{\int_0^t \zeta^g_s \PP^G(dg) dW_s + \int_0^t \psi^{g}_s \PP^G(dg) d\tilde N_s}{\PP^G(dg)}\ .
    \end{align*}
    By the orthogonality of $W$ and $\tilde N$ we deduce that 
    $\langle W,\tilde N\rangle_t^\bF = 0$ 
    and the result follows by applying the same reasoning as in Lemma \ref{pois.drift}.
\end{proof}
\begin{theorem}\label{alpha.binary.compound}
  If $G$ is a binary random variable, then
  $$ \alpha_t^G = \EE\left[D_t G\vert \cF_t\right]\frac{G-\EE[G\vert \cF_t]}{\VV[G\vert \cF_t]}\ ,\quad \gamma_t^G = \EE\left[D_{t,1} G\vert \cF_t\right] \frac{G-\EE[G\vert \cF_t]}{\VV[G\vert \cF_t]}\ .$$
\end{theorem}
When we consider a $\bG$-agent playing with
$\pi\in\cA(\bG)$, the \Ito integral fails for both the Brownian motion and the Poisson process.
Using the \emph{Jacod Hypothesis} and the 
Theorem \ref{alpha.binary.compound} we can define the following 
$\bG$-local martingales
\begin{equation}
    \widehat{W}_t:= W_t-\int_0^t \alpha^G_s ds\ ,\quad \widehat{N}_t:= N_t - \int_0^t \lambda_s(1+\gamma_s^G)ds\ ,\quad 0\leq t \leq T\ .
\end{equation}
The dynamics of the wealth process satisfy the following SDE,
\begin{equation}\label{new.dynamics.compound}
        \frac{dX_t^\pi}{X_t^{\pi}} = \left((1-\pi_t)\rho_t+ \pi_t\mu_t + \pi_t\sigma_t\alpha_t^G+ \pi_t\theta_t\lambda_t\gamma_t^G \right)dt + \pi_t\sigma_t d\widehat{W}_t + \pi_t\theta_t d\widehat{N}_t \ ,\quad X_0 = x_0 \ ,
\end{equation}
and we get the following explicit solution
\begin{align}\label{solution.forward.X}
\ln \frac{X_T^\pi}{x_0} =& \int_0^T \rho_s + \pi_s(\mu_s-\rho_s+\sigma_s\alpha_s^G)- \frac{1}{2}\pi^2_s\sigma^2_s + \lambda_s(1+\gamma_s^G)\ln(1+\pi_s\theta_s)-\lambda_s\pi_s\theta_sds\notag \\
&+\int_0^T\pi_s\sigma_sd\widehat{W}_s+ \int_0^T \ln(1+\pi_s\theta_s) d\widehat{N}_s\ .
\end{align}
Finally, using the integrability conditions, we can compute the expectation of the stochastic integrals and we get,
\begin{align*}
\EE\left[\ln \frac{X^{\pi}_T}{x_0}\right] =& \int_0^{T}\EE\left[ \rho_s + \pi_s (\mu_s-\rho_s+\alpha^G_s\sigma_s)- \frac{1}{2}\pi^2_s\sigma^2_s\right.\\ 
&+ \left. \lambda_s(1+\gamma_s^G)\ln(1+\pi_s\theta_s)-\lambda_s\pi_s\theta_s \right]ds\ .
\end{align*}
\begin{proposition}
The optimal strategy of the problem given by \eqref{optimization.problem} with G binary, information flow $\bG$ and both Brownian and Poisson noises is given by,
\begin{align}\label{optimal.portfolio.compound}
    \pi_s =& \frac{1}{2}\left(\frac{\mu_s-\rho_s+\alpha_s^G\sigma_s-\lambda_s\theta_s}{\sigma_s^2}-\frac{1}{\theta_s} \right)\notag\\
    &+\sgn(\theta_s)\frac{1}{2} \sqrt{\left(\frac{\mu_s-\rho_s+\alpha_s^G\sigma_s-\lambda_s\theta_s}{\sigma_s^2}+\frac{1}{\theta_s}\right)^2+ 4\lambda_s\frac{1 + \gamma^G_s}{\sigma_s^2}}
\end{align}
\end{proposition}
\begin{proof}
    We proceed with a perturbation argument.
    Let $\beta\in\cA(\bG)$ a bounded strategy and let $\epsilon >0$.
    We define
    \begin{align*}
        I(\epsilon) :=& \EE\left[\int_0^T \rho_s + (\pi_s+\epsilon\beta_s)  (\mu_s-\rho_s+\alpha^G_s\sigma_s)- \frac{1}{2}(\pi_s+\epsilon\beta_s)^2_s\sigma^2_s \, ds\right]\\
        &+ \EE\left[\int_0^T \lambda_s(1+\gamma_s^G)\ln(1+(\pi_s+\epsilon\beta_s)\theta_s)-\lambda_s(\pi_s+\epsilon\beta_s)\theta_s\, ds\right]
    \end{align*}
    and we consider the first order condition $0 = I'(0)$ as follows,
    \begin{equation}\label{opt.cond.mixed}
        0 = %\lim_{\epsilon\to 0}\frac{I(\epsilon)-I(0)}{\epsilon} 
        \EE\left[\int_0^T\beta_s\left( \mu_s-\rho_s+\alpha_s^G\sigma_s - \pi_s\sigma_s^2 + \lambda_s(1+\gamma_s^G)\frac{\theta_s}{1+\pi_s\theta_s} - \lambda_s\theta_s \right) ds\right]\ .
    \end{equation}
    Then we take $\beta_s = \xi\bOne_{\{t\leq s< t+h\}}$ with $\xi$ a $\cG_t$-measurable and bounded random variable.
    We can rewrite the Equation \eqref{opt.cond.mixed} in terms of conditional expectation as follows,
    $$0 = \EE\left[\int_t^{t+h}\left( \mu_s-\rho_s+\alpha^G_s\sigma_s - \pi_s\sigma_s^2 + \lambda_s(1+\gamma^G_s)\frac{\theta_s}{1+\pi_s\theta_s} - \lambda_s\theta_s \right) ds\vert \cG_t\right] \ ,$$
    for $0 < h < T-t$ arbitrarily near to zero, finally we get
    $$0 = \mu_s-\rho_s+\alpha^G_s\sigma_s - \pi_s\sigma_s^2 + \lambda_s(1+\gamma^G_s)\frac{\theta_s}{1+\pi_s\theta_s} - \lambda_s\theta_s\ . $$
    To short the notation, we define the terms $d_s := \mu_s-\rho_s+\alpha_s^G\sigma_s-\lambda_s\theta_s$ and \hbox{$c_s := \lambda_s(1+\gamma^G_s)$}.
    Then we derive the next equation,
    $$ 0 = \left(\sigma_s^2\theta_s\right) \pi_s^2 + \left(\sigma_s^2-\theta_sd_s\right)\pi_s - \left(d_s+\theta_sc_s\right)$$
    with the following solutions,
    \begin{align*}
        \pi^\pm_s &= \frac{\theta_s d_s - \sigma_s^2\pm \sqrt{\left( \theta_s d_s-\sigma_s^2 \right)^2 + 4\sigma_s^2\theta_s\left(d_s+\theta_sc_s\right)}}{2\sigma_s^2\theta_s}\\
        &= \frac{1}{2}\left(\frac{d_s}{\sigma_s^2}-\frac{1}{\theta_s} \right) \pm\frac{1}{2} \sqrt{\left(\frac{d_s}{\sigma_s^2}+\frac{1}{\theta_s}\right)^2+4 \frac{c_s}{\sigma_s^2}}\ ,
    \end{align*} 
    where in the last step we have used arithmetic computations.
    It can be verified that $I''(0)<0$ and the pair of strategies $\pi^\pm$ are maximum if and only if they are admissible.
    Then we need to check if the condition $1+\pi^{\pm}_s\theta_s > 0$ is satisfied.
    We rewrite the pair as follows,
    \begin{align*}
        1 + \pi_s\theta_s &= \frac{1}{2}\left( \frac{d_s\theta_s}{\sigma_s^2} + 1 \right) \pm \frac{1}{2} \frac{\theta_s}{\abs{\theta_s}}\sqrt{\left( \frac{d_s\theta_s}{\sigma_s^2} + 1 \right)^2 + 4\theta_s^2\frac{\lambda_s(1+\gamma_s^G)}{\sigma_s^2}}\\
        &= \frac{1}{2}\left( \frac{d_s\theta_s}{\sigma_s^2} + 1 \right) \pm \frac{1}{2} \sgn(\theta_s)\sqrt{\left( \frac{d_s\theta_s}{\sigma_s^2} + 1 \right)^2 + 4\theta_s^2\frac{\lambda_s(1+\gamma_s^G)}{\sigma_s^2}}
    \end{align*}
    We use the following fact from real analysis
    $$ f^+(x) = x + \sqrt{x^2 + a} > 0\ ,\quad  f^-(x) = x - \sqrt{x^2 + a} < 0 \ ,\quad \forall (x,a)\in \bR\times\bR^+\ ,$$
    and we deduce that on the set \hbox{$\{\theta_s > 0\}$} the unique optimal solution is $\pi^+$ and on 
    \hbox{$\{\theta_s < 0\}$} is $\pi^-$, $\forall s\in[0,T]$, then the result holds.
\end{proof}
\begin{remark}
Note that in the strategy \eqref{optimal.portfolio.compound} we find the usual Merton strategy for Brownian noise and the additional information,
$$ \pi_s^M = \frac{\mu_s-\rho_s+\alpha^G_s\sigma_s}{\sigma^2_s}\ . $$
The optimal strategy in the mixed market include the 
Poisson distortion with $\pm 1/\theta_s$ and the joint effect of the additional information on the Poisson process with the Brownian process, i.e.,
$$ \pi_s = \frac{1}{2}\left(\pi_s^M-\lambda_s\frac{\theta_s}{\sigma^2_s}-\frac{1}{\theta_s} \right)+ \sgn(\theta_s)\frac{1}{2} \sqrt{\left(\pi_s^M-\lambda_s\frac{\theta_s}{\sigma^2_s}+\frac{1}{\theta_s}\right)^2+ 4\lambda_s\frac{1+ \gamma^G_s}{\sigma_s^2}}\ . $$
\end{remark}
\begin{example}\label{Example.3}
Let $A = (-\infty,a]$ and $B=(-\infty,b]$ be two half-bounded intervals.
We define the binary random variable as the following product indicator,
$$ G =  \bOne_{\{W_T\leq a\}\times\{N_T\leq b\}} = \bOne_{\{W_T\leq a\}}\bOne_{\{N_T\leq b\}}\ . $$
In order to achieve explicit results, we assume that the intensity satisfies \hbox{$\lambda_t\in\cF_0$}, \hbox{$\forall t\in[0,T]$}, because in the most general case we can not compute explicitly the probabilities.

According to \cite{LeonSoleUtzetVives2002}, thanks to the independence of the Brownian motion and the Poisson process, the Malliavin derivatives in each direction can be easily computed as follows,
$$D_t G = \bOne_{\{N_T\leq b\}} D_t \bOne_{\{W_T\leq a\}}\ ,\quad D_{t,1} G = \bOne_{\{W_T\leq a\}} D_{t,1} \bOne_{\{N_T\leq b\}} $$
so we need to calculate the conditional expectation of these terms.
\begin{align*}
    \EE\left[D_{t,1} G\vert \cF_t\right] &= \EE\left[  \bOne_{\{W_T\leq a\}} D_{t,1} \bOne_{\{N_T\leq b\}}\vert \cF_t\right] = - \EE\left[  \bOne_{\{W_T\leq a\}}\bOne_{\{N_T = b\}}\vert \cF_t\right] \\
%    &= \EE\left[\bOne_{\{W_T\leq a\}}\left(\bOne_{\{N_T + 1 \leq b \}} - \bOne_{\{N_T \leq b \}}\right)\vert \cF_t\right] \\
%    &= \EE\left[\bOne_{\{W_T\leq a\}\times\{N_T + 1 \leq b \}}\vert \cF_t\right] - \EE\left[\bOne_{\{W_T\leq a\}\times\{N_T  \leq b \}}\vert \cF_t\right] \\
%    &= \PP(\{W_T\leq a\}\times\{N_T + 1 \leq b \}\vert \cF_t) -  \PP(\{W_T\leq a\}\times\{N_T \leq b \}\vert \cF_t)\\
%    &= \PP(W_T\leq a\vert \cF_t)\PP(N_T + 1 \leq b \vert \cF_t) -  \PP(W_T\leq a\vert \cF_t)\PP(N_T \leq b \vert \cF_t)\\
    &= - \EE\left[  \bOne_{\{W_T\leq a\}\times\{N_T = b\}}\vert \cF_t\right] =  - \PP\left(  \{W_T\leq a\}\cap\{N_T = b\}\vert \cF_t\right)\\
    &= -\left(\int_{-\infty}^{a}\frac{\exp\left(-\frac{(x-W_t)^2}{2(T-t)}\right)}{\sqrt{2\pi(T-t)}} dx\right) \left( e^{-\Lambda(t,T)}\frac{(\Lambda(t,T))^{b-N_t}}{(b-N_t)!} \bOne_{\{N_t\leq b\}}\right)\\
    \EE\left[D_t G \vert \cF_t\right] &= \EE\left[\bOne_{\{N_T\leq b\}} D_t \bOne_{\{W_T\leq a\}}\vert \cF_t\right] =  \EE\left[\bOne_{\{N_T\leq b\}} \delta_a(W_T)\vert \cF_t\right]\ ,
\end{align*}
where in the first computation we have used our Example \ref{Example.1} and for the second one we refer to \cite{Bermin02} for the generalized Malliavin derivative of the indicator function.
In order to compute the conditional expectation, we consider the following conditional distribution function,
$$ F(x,y):=\PP(W_T\leq x,N_T\leq y\vert \cF_t) = \int_{-\infty}^x \sum_{k = 0}^y f_{W_T\vert W_t}(u) p_{N_T\vert N_t}(k) du\ , $$
where $f_{W_T\vert W_t}(u)$ denotes the density function of $(W_T\vert W_t)$ and $p_{N_T\vert N_t}(k)$ the probability function of $(N_T\vert N_t)$.
Both of them are well-known.
Then,
\begin{align*}
    \EE\left[D_t G \vert \cF_t\right] &=  \int_{-\infty}^\infty \sum_{k = 0}^\infty f_{W_T\vert W_t}(u) p_{N_T\vert N_t}(k) \bOne_{\{k\in B\}} \delta_a(u) du\\ 
    &=  \int_{-\infty}^\infty  f_{W_T\vert W_t}(u)  \delta_a(u) du \sum_{k = 0}^\infty p_{N_T\vert N_t}(k) \bOne_{\{k\in B\}}\\
    &= f_{W_T\vert W_t}(a)  \sum_{k = 0}^b p_{N_T\vert N_t}(k) \\
    &= \frac{\exp\left(-\frac{(a-W_t)^2}{2(T-t)}\right)}{\sqrt{2\pi(T-t)}} \sum_{k = 0}^{b-N_t} e^{-\Lambda(t,T)}\frac{(\Lambda(t,T))^k}{k!}\bOne_{\{N_t\leq b\}}
\end{align*}
Finally, we deduce the PRP via Clark-Ocone formula,
\begin{align*}
&\bOne_{\{W_T\leq a\}\times\{N_T\leq b\}} = \PP(W_T\leq a)\PP(N_T\leq b) \\
&+ \int_0^T \left(\frac{\exp\left(-\frac{(a-W_t)^2}{2(T-t)}\right)}{\sqrt{2\pi(T-t)}}\right) \left(\sum_{k = 0}^{b-N_t} e^{-\Lambda(t,T)}\frac{(\Lambda(t,T))^k}{k!} \right)  dW_t\\
&-\int_0^T\left(\int_{-\infty}^{a-W_t}\frac{\exp\left(-\frac{(x-a+W_t)^2}{2(T-t)}\right)}{\sqrt{2\pi(T-t)}} dx\right) \left( e^{-\Lambda(t,T)}\frac{(\Lambda(t,T))^{b-N_t}}{(b-N_t)!} \bOne_{\{N_t\leq b\}}\right)d\tilde N_t
\end{align*}
from which the processes $\alpha^G=\prT{\alpha^G}$ and $\gamma^G=\prT{\gamma^G}$ appearing in Theorem \ref{alpha.binary.compound} are determined.
\end{example}
\begin{example}\label{Example.4}
Let's define 
\begin{equation}
    M_{s,t}:=\sup_{s\leq u \leq t} W_u\ ,\quad J_{s,t}:=\sup_{s\leq u \leq t} \tilde N_u\ ,
\end{equation}
and $M_t:= M_{0,t}$ and $J_t:= J_{0,t}$. 
To short the notation, we define the intervals \hbox{$A:=(a_1,a_2]$} and $B=(b_1,b_2]$.
We consider the following example
\begin{equation}
    G = \bOne_{\{M_T\in A\}\times \{J_T\in B\}} = \bOne_{\{M_T\in A\}}\bOne_{\{J_T\in B\}}
\end{equation}
and we proceed as before.
\begin{align*}
    D_t G &= \bOne_{\{J_T\in B\}} D_t \bOne_{\{M_T\in A\}} = \bOne_{\{J_T\in B\}} D_t \left(\bOne_{\{M_T\leq a_2\}} - \bOne_{\{M_T \leq a_1\}}\right) \\
    &= \bOne_{\{J_T\in B\}} \bOne_{\{M_t\leq M_{t,T}\}}\left(-\delta_{a_2}(M_T) + \delta_{a_1}(M_T)\right) 
    \ ,
\end{align*}
we refer to \cite{Bermin02} for a detailed explanation of the Malliavin derivative of the running maximum $M_T$.
We consider the conditional expectation of the 
$D_t G$ after splitting thanks to the independence.
Then,
\begin{align}\label{ex.max.alpha1}
    \EE\left[D_t \bOne_{\{M_T\in A\}} \vert \cF_t\right] &= \EE\left[\bOne_{\{M_t\leq M_{t,T}\}}\left(-\delta_{a_2}(M_{t,T}) + \delta_{a_1}(M_{t,T})\right) \vert \cF_t\right]\notag\\
    &= \int_0^{+\infty} \bOne_{\{M_t\leq m\}}\left(-\delta_{a_2}(m) + \delta_{a_1}(m)\right) f_{t}(m) dm\notag\\
    &= \bOne_{\{M_t\leq a_1\}} f_{t}(a_1) - \bOne_{\{M_t\leq a_2\}} f_{t}(a_2)\ ,
\end{align}
being $f_t$ the density of the random variable $M_{t,T}$ given $\cF_t$, which is equivalent to consider the variable $M_{T-t}$ in the domain $(W_t,+\infty)$, i.e.,
$$ f_t(m) =\frac{2 e^{-\frac{(m-W_t)^2}{2(T-t)}}}{ \sqrt{2\pi (T-t)}}\ ,\quad m\geq W_t\ . $$
On the other hand we compute the conditional expectation of the remained Poisson term,
\begin{align*}
    \EE\left[\bOne_{\{J_T\in B\}}\vert \cF_t\right] &=  \PP\left(J_T\in B\vert \cF_t\right) =  \PP\left(\max\{J_t,J_{t,T}\}\in B\vert \cF_t\right)\\
    &= \PP\left(J_t + (J_{t,T}-J_t)^+\in B\vert \cF_t\right)\\
    &= \PP\left((J_{T-t}-b_t)^+\in (b_1-J_t,b_2-J_t] \right)\ ,
\end{align*}
where $b_t:= J_t-\tilde N_t$ and using that $J_{t,T}-\tilde N_t$ is independent of $\cF_t$. 
We aim to compute
\begin{equation*}
     \EE\left[\bOne_{\{J_T\in B\}}\vert \cF_t\right] =  \PP\left((J_{T-t}-b_t)^+ >  b_1-J_t\right) -  \PP\left((J_{T-t}-b_t)^+ > b_2-J_t\right)\ .
\end{equation*}
Each one of the probabilities can be computed as
\begin{align*}
    \PP\left((J_{T-t}-b_t)^+ >  b_1-J_t\right) &= \bOne_{\{b_1 - J_t\leq 0\}} + \bOne_{\{b_1 - J_t > 0\}}\overbar{F}^N_{T-t}(b_1-\tilde N_t) \\
    &= 1+ \bOne_{\{b_1 - J_t > 0\}}\left(\overbar{F}^N_{T-t}(b_1-\tilde N_t)-1\right)\\
    \PP\left((J_{T-t}-b_t)^+ >  b_2-J_t\right) &= \bOne_{\{b_2 - J_t\leq 0\}} + \bOne_{\{b_2 - J_t > 0\}}\overbar{F}^N_{T-t}(b_2-\tilde N_t)\\
    &= 1+ \bOne_{\{b_2 - J_t > 0\}}\left(\overbar{F}^N_{T-t}(b_2-\tilde N_t)-1\right)
\end{align*}
where the survival function is defined as $\overbar{F}^N_{T-t}(x) = \PP(J_{T-t}>x)$ for every $x\geq 0$. 
See~\cite{Kuznetsov10} for an explicit computation of the distribution of the running supremum.
In terms of the distribution function $F^N$ it can be simplified as follows,
\begin{align}\label{ex.max.alpha2}
    \EE\left[\bOne_{\{J_T\in B\}}\vert \cF_t\right] =&  \bOne_{\{b_1 - J_t > 0\}}\left(\overbar{F}^N_{T-t}(b_1-\tilde N_t)-1\right) \\
    &- \bOne_{\{b_2 - J_t > 0\}}\left(\overbar{F}^N_{T-t}(b_2-\tilde N_t)-1\right)\notag\\
    =& \bOne_{\{b_2 - J_t > 0\}}F^N_{T-t}(b_2-\tilde N_t)-\bOne_{\{b_1 - J_t > 0\}}F^N_{T-t}(b_1-\tilde N_t)\ .
\end{align}
Then, taking into account \eqref{ex.max.alpha1} and \eqref{ex.max.alpha2} the process $\alpha^G$ is fully determined.
We proceed in the same way in order to compute $\EE[D_{t,1}G\vert \cF_t]$.
Using the operator~$\Psi$, we can compute the following Malliavin derivative
\begin{equation*}
    D_{t,1} \bOne_{\{J_T\in B\}} =\bOne_{\{\max{\{J_t,1+J_{t,T}\}}\in B\}} - \bOne_{\{J_T\in B\}}
\end{equation*}
where the second term has been calculated in \eqref{ex.max.alpha2}.
For the first one we have
\begin{align*}
    \EE[&\bOne_{\{\max{\{J_t,1+J_{t,T}\}}\in B\}}\vert \cF_t] = \PP(\max{\{J_t,1+J_{t,T}\}}\in B\vert \cF_t) \\
    &= \bOne_{\{b_2 - J_t > 0\}}F^N_{T-t}(b_2-\tilde N_t-1)-\bOne_{\{b_1 - J_t > 0\}}F^N_{T-t}(b_1-\tilde N_t-1)
\end{align*}
where we have omitted some steps as they were similar to ones shown before.
Finally
\begin{equation}\label{ex.max.gamma1}
    \EE\left[\bOne_{\{M_T\in A\}}\vert \cF_t\right] =
    \bOne_{\{a_2 - M_t > 0\}}F^W_{T-t}(a_2- W_t)-\bOne_{\{a_1 - M_t > 0\}}F^W_{T-t}(a_1- W_t)\ ,
\end{equation}
where in this case $\overbar{F}_t^W(y) = 2(1-\Phi(y/\sqrt{t}))$ and again the process $\gamma^G$ is determined.
\end{example}
\section*{Conclusion}
In this paper we show how to incorporate anticipative information 
in a filtration generated by a Brownian motion and a Poisson process.
We compute the compensators in a general framework of additional information (see Lemma~\ref{mixed.compens}), 
and then we focus on the binary case to consider more explicit examples (see Theorem~\ref{alpha.binary.compound}).
In particular, 
we study the case in which a $\bG$-agent knows if the final pair of random variables $(W_T,N_T)$ 
are within a certain rectangular region, as well as the case that considers a similar type of information about the pair of running maximums~$(M_T,J_T)$, see Examples~\ref{Example.3} and~\ref{Example.4}.

When the dynamics of the risky asset dynamics are driven by the Poisson process only, 
we give the exact value of the additional information in terms of an entropy similarly to the corresponding continuous case, see Theorem~\ref{theo.gain.pois} and compare it with~\cite{AmendingerImkellerSchweizer1998}.

%% BIBLIOGRAPHY
\bibliography{main.bib} 

\section*{Acknowledgments}
This research was partially supported by the Spanish \emph{Ministerio de Economía y Competitividad} grant
%s MTM2017-85618-P (via FEDER funds) and 
PID2020-116694GB-I00.
The first author acknowledges financial support by the Community of Madrid  within the framework of the multi-year agreement with the Carlos III Madrid University in its line of action ``\emph{Excelencia para el Profesorado Universitario}'' (V Plan Regional de Investigación Científica e Innovación Tecnológica 2016-2020).
The second author acknowledges financial support by an FPU Grant (FPU18/01101) of Spanish \emph{Ministerio de Ciencia, Innovaci\'{o}n y Universidades}.

\end{document}